\documentclass[12pt]{amsart}
\usepackage{amssymb}
\input amssym.def
\usepackage{amsmath, amsfonts,hyperref,xcolor}
\usepackage{amscd}
\usepackage[mathscr]{eucal}
\hypersetup{colorlinks=true,citecolor={purple},linkcolor={teal},urlcolor={violet}}

\newcommand{\N}{{\mathbb N}}
\newcommand{\Z}{{\mathbb Z}}
\newcommand{\Q}{{\mathbb Q}}
\newcommand{\C}{{\mathbb C}}
\newcommand{\R}{{\mathbb R}}

\newcommand{\z}{\zeta}

\newtheorem{thm}{Theorem}
\newtheorem{lem}{Lemma}
\newtheorem{cor}{Corollary}
\newtheorem{prop}{Proposition}
\newtheorem{rmk}{Remark}
\newtheorem{defn}{Definition}
\newcommand{\thmref}[1]{Theorem~\ref{#1}}
\newcommand{\propref}[1]{Proposition~\ref{#1}}
\newcommand{\lemref}[1]{Lemma~\ref{#1}}

\newcommand{\rmkref}[1]{Remark~\ref{#1}}

\begin{document}

\title[Multiple Lerch zeta functions]
{Multiple Lerch zeta functions and an 
idea of Ramanujan}

\author{Sanoli Gun and Biswajyoti Saha}

\address{Sanoli Gun\\ \newline
Institute of Mathematical Sciences, C.I.T. Campus, Taramani, 
Chennai, 600 113, India.}
\email{sanoli@imsc.res.in}

\address{Biswajyoti Saha\\ \newline
Tata Institute of Fundamental Research,
Homi Bhabha Road, Navy Nagar, Colaba,
Mumbai, 400 005, India.}
\email{biswa@math.tifr.res.in}



\begin{abstract}
In this article, we derive meromorphic continuation
of multiple Lerch  zeta functions by generalising an 
elegant identity of Ramanujan. Further, we
describe the set of all possible singularities of 
these functions. Finally,  for 
the multiple Hurwitz zeta functions, we
list the exact set of singularities. 
\end{abstract}

\maketitle

\section{Introduction and statements of the theorems}

In 1917, Ramanujan introduced a novel idea
which enabled him to derive an elegant functional
equation of the classical Riemann zeta function.
He showed that for $\Re(s)>1$,
the Riemann zeta function
$\zeta(s):=\sum_{n \ge 1} \frac{1}{n^s}$
satisfies the following formula:
\begin{equation}\label{rama}
1=\sum_{k\ge 0} (s-1)_k (\z(s+k)-1),
\end{equation}
where the right hand side of \eqref{rama} converges 
normally on any compact subset of $\Re(s) > 1$
and
$$
(s)_k := \frac{s \cdots (s+k)}{(k+1)!}
$$
for any $k \ge 0$ and $s \in \C$.
An elementary proof of this formula,
as suggested by Ecalle \cite{JE},
can be deduced from the identity
$$
(n-1)^{1-s}-n^{1-s}= \sum _{k\ge 0} (s-1)_k \ n^{-s-k},
$$
which is valid for any natural number $n \ge 2$
and any $s \in \C$. 

In fact, Ecalle \cite{JE} also suggested how one
can derive a formula similar to \eqref{rama},
for the multiple zeta functions. 
Following Ecalle's indication, the last author
along with Mehta and Viswanadham \cite{MSV} 
derived such a formula for the multiple zeta functions
and studied the meromorphic continuations
as well as the set of polar singularities
of them (see \cite{MSV} and \cite{JO} for details).

Meromorphic continuations of the multiple zeta functions
was proved first by Zhao~\cite{JZ}. Around the same time,
Akiyama, Egami and Tanigawa \cite{AET} gave an alternate
proof of meromorphic continuations along with the exact set
of polar hyperplanes for these functions.  In~\cite{MSV},
the last author along with Mehta and Viswanadham
introduced the method of matrix formulation
to write the down the residues of the multiple
zeta functions in a computable form, and thereby
reproved the theorem of Akiyama, Egami and Tanigawa.  

In this paper, we generalise the identity of Ramanujan 
to obtain meromorphic continuations
as well as the set of possible singularities
of the multiple Lerch zeta functions (defined below). 
When $r=1$, it was done by the last author in \cite{BS}.
Let $r > 0 $ be a natural number 
and $U_r$ be an open subset of $\C^r$ defined
as follows:
$$
U_r := \{ (s_1, \ldots, s_r) \in \C^r ~|~ \Re(s_1 + \cdots + s_i) > i
~\text{ for all }~ 1 \le i  \le r \}.
$$
Then for real numbers $ \lambda_1, \ldots, \lambda_r,
\alpha_1, \ldots, \alpha_r  \in [0,1)$
and complex $r$-tuples $(s_1, \ldots, s_r) \in U_r$,
the multiple Lerch zeta function of depth $r$ is defined by
\begin{equation}\label{lerch}
L_r ( \lambda_1, \ldots, \lambda_r;
 \alpha_1, \ldots, \alpha_r ;  s_1, \ldots, s_r )
:= 
\sum_{n_1>\cdots>n_r>0} \frac{e(\lambda_1 n_1)
 \cdots e(\lambda_r n_r)}
{(n_1+\alpha_1)^{s_1} \cdots (n_r + \alpha_r)^{s_r}},
\end{equation}
where $e(a) := e^{2 \pi \iota a}$ for $a \in \R$.
The series on the right hand side of \eqref{lerch}
is normally convergent on compact subsets of $U_r$
(see \propref{norm}) and hence defines a holomorphic
function there.

Before we state our theorems, let us introduce few more notations.
For integers $1 \le i \le r$ and $k \ge 0$, let
\begin{equation*}
H_{i,k} ~:=~ \{ (s_1, \ldots, s_r) \in \C^r ~|~
 s_1 + \cdots + s_i = i - k \}.
\end{equation*}
Also for $1 \le i \le r$, let
\begin{equation*}
\mu_i ~:=~ \sum_{j=1}^i \lambda_j
\end{equation*}
and $\Z_{\le j}$ denote the set of integers
less than and equal to $j$.
In this article we prove the following theorems.

\begin{thm}\label{easy}
Assume that $\mu_i \not\in \Z$ for all $1 \le i \le r$. Then
$L_r(\lambda_1,\ldots,\lambda_r; 
\alpha_1,\ldots,\alpha_r; s_1,\ldots,s_r)$ can be extended
analytically to the whole of $\C^r$. 
\end{thm}

\begin{rmk}\label{easy-kn}\rm
If $r=1$ and $\lambda_1 \notin \Z$,
Lerch \cite{ML} showed that
$L_1(\lambda_1;\alpha_1;s_1)$ can be
extended to an entire function of $\C$. 
\end{rmk}


\begin{thm}\label{multiple-Lerch}
With the notations as above, let $i_1 < \cdots < i_{m}$ be the 
only indices for which $\mu_{i_j} \in \Z,~ 1 \le j \le m$.
\begin{itemize}
\item
If $i_1 =1$, then 
$L_r(\lambda_1,\ldots,\lambda_r;  
\alpha_1,\ldots,\alpha_r; s_1,\ldots,s_r)$
can be meromorphically continued to $\C^r$
with possible simple poles along the hyperplanes 
$$
H_{1,0} \ \text{ and } \
H_{i_j, k} ~\text{ for  }~
2 \le j \le m \ \text{ with }~(i_j -k) \in \Z_{\le j}.
$$

\item
If $i_1 \ne 1$, then $L_r(\lambda_1,\ldots,\lambda_r;  \alpha_1,
\ldots,\alpha_r; s_1,\ldots,s_r)$ 
can be meromorphically continued to $\C^r$
with possible simple poles along the hyperplanes 
$$
H_{i_j, k} ~\text{ for  }~
1 \le j \le m \ \text{ with }~(i_j -k) \in \Z_{\le j}.
$$ 
\end{itemize}
\end{thm}

\begin{rmk}\label{multiple-Lerch-kn}\rm
\thmref{multiple-Lerch} is well known in
the special case when $r=1$.
In this case, $L_1(\lambda_1; \alpha_1, s_1)$ where
$\lambda_1 \in \Z$ is essentially the Hurwitz zeta function
and hence can be extended analytically to $\C$, except at $s=1$, 
where it has a simple pole with residue $1$.
\end{rmk}

Komori \cite{YK} considered certain several variable 
generalisations of the Lerch zeta function and derived
meromorphic continuations 
of these functions through integral representation.
He also obtained a rough estimation of their
possible singularities (see \cite{YK}, \S 3.6).

Now if we choose $\lambda_i = 0$ for 
$1 \le i \le r$ in \eqref{lerch}, then we get
$$
L_r (0, \ldots, 0; \alpha_1, \ldots, \alpha_r;
s_1, \ldots, s_r)
= 
\z_r(s_1, \ldots, s_r; \alpha_1, \ldots, \alpha_r),
$$
the multiple Hurwitz zeta function of depth $r$,
and further if $\alpha_i =0$ 
for $1 \le i \le r$, we get
$$
L_r (0, \ldots, 0;  0, \ldots, 0; s_1, \ldots, s_r)
= 
\z_r(s_1, \ldots, s_r),
$$
the multiple zeta function of depth $r$.

Akiyama and Ishikawa \cite{AI} obtained the
meromorphic continuation of the multiple
Hurwitz zeta functions together with their possible
polar singularities. In the special case when 
$\alpha_i \in \Q$ for $1 \le i \le r$,
they also derived the exact set of singularities.
This has also been done \cite{MV}.
Using Mellin-Barnes integral formula, Matsumoto \cite{KM}
 showed meromorphic continuation of 
 multiple Hurwitz zeta functions with possible set of singularities.
 Finally we refer to the interested reader the following papers, namely
 \cite{FKMT} and \cite{TO} where similar themes
 are addressed. 
An expression for residues along these possible polar hyperplanes
were obtained in \cite{KM,MV}.
For the multiple Hurwitz zeta functions,
we are now able to characterise the exact
set of singularities. This complete characterisation
is new.  More precisely, we have the following theorem.

\begin{thm}\label{multiple-Hurwitz}
The multiple Hurwitz zeta function 
$\z_r(s_1, \ldots, s_r;\alpha_1, \ldots, \alpha_r)$
has meromorphic continuation to $\C^r$. Further,
all its poles are simple and they are
along the hyperplanes 
$$
H_{1,0}
\phantom{m}\text{and}\phantom{m}
~H_{i,k}~ 
\text{ for }~
2 \le i \le r, ~k \ge 0
$$
except when $i =2, ~k \in K$, where
$$
K := \{ n \in \N ~|~  B_{n}(\alpha_2 - \alpha_1)=0 \}
$$
and $B_{n} (t)$ denotes the $n$-th Bernoulli polynomial defined by
generating series
$$
\frac{xe^{tx}}{e^x -1} = \sum_{ n \ge 0} B_{n}(t)\frac{x^{n}}{ n!}.
$$
\end{thm}

 Before proceeding further, we indicate, compare and contrast some 
 of the other existing works vis-{\`a}-vis our  work.
 In \cite{MS}, the authors obtain meromorphic continuation for  
 multiple Hurwitz zeta function of an arbitrary depth $r$ using 
 Binomial expansion. In order to do so, they deduce 
 a functional equation involving various Multiple Hurwitz zeta 
 functions of a fixed depth $r$ (see Theorem 5.2).  The novelty of
 our work is to deduce a functional equation involving 
 Multiple Hurwitz zeta functions of depth $r$ with Multiple
 Hurwitz zeta functions of depth $r-1$ (see Theorem 4).  
 This is the crucial ingredient which enables us to derive information 
 about the poles and residues of such functions,  which was not
 done in \cite{MS}.  The use of Binomial expansion has
 also been exploited in \cite{DE} for 
 proving the meromorphic continuation of multiple Hurwitz zeta
 functions. More precisely, he uses products
 of Binomial expansions which we avoid. 
 Also he deals only with the 
 diagonal vectors
 in the $r$-dimensional complex plane  while we allow 
 arbitrary vectors in $\C^r$. Furthermore, the author does
 not deal with the poles and residues of these functions.

 The paper is distributed as follows. In the next section, 
 we prove some intermediate 
 results and derive functional identities for the multiple
 Lerch zeta function which is a generalisation
 of the identity of Ramanujan (see \thmref{gen-rama}). 
 In Section 3, we derive 
 meromorphic continuation of the multiple Lerch
 zeta functions as well as their possible set
 of singularities using these functional identities.
 In Section 4, we follow \cite{MSV} to write down the
 relevant functional identity for the multiple
 Hurwitz zeta functions in terms of infinite matrices, in order to
 obtain an expression for residues along the singular
 hyperplanes (see \thmref{residues-mhzf}).
 Finally in Section 5, we complete the
 proof of \thmref{multiple-Hurwitz}. For this we need
 to use some fundamental properties of the zeros of the
 Bernoulli polynomials. These results are discussed
 in  \S5.1.

\section{Intermediate results and generalised Ramanujan's identity}

In this section, we derive an analogue of \eqref{rama}
(see \eqref{trans} below) for the multiple Lerch zeta functions.
In order to establish \eqref{trans} we need some intermediate results.
Before we state our theorem, we start with the notion of normal
convergence.

\begin{defn}
Let $X$ be a set and $(f_i)_{i \in I}$ be a 
family of complex valued functions
defined on $X$. We say that the family of 
functions $(f_i)_{i \in I}$
is normally summable on $X$ or the series $\sum_{i \in I} f_i$ 
converges normally on $X$ if 
$$
\|f_i\|_X := \sup_{x \in X} |f(x)| < \infty ,
~\text{      for all  }i \in I
$$ 
and the family of real numbers 
$(\| f_i \|_X)_{i \in I}$ is summable. 
\end{defn}

%
%
%
%
%
\begin{defn}
Let $X$ be an open subset of $\C^r$ and $(f_i)_{i \in I}$ be a
family of meromorphic functions on $X$. We say that
$(f_i)_{i \in I}$ is normally summable or $\sum_{ i \in I} f_i$
is normally convergent on all compact
subsets of $X$ if for any compact subset $K$ of $X$,
there exists a finite set $J \subset I$ such that
each $f_i$ for $i \in I \setminus J$ is holomorphic in an open 
neighbourhood of $K$ and the family 
$(f_i|K)_{ i \in {I \setminus J}}$ is normally summable
on $K$. In this case, $\sum_{ i \in I} f_i$ is a well
defined meromorphic function on~$X$.
\end{defn}

We now have the following theorem.

\begin{thm}\label{gen-rama}
Let $r \ge 2$ be a natural number, $\lambda_1, \ldots, \lambda_r, 
\alpha_1, \ldots, \alpha_r \in [0,1)$.
Then for any $(s_1, \ldots, s_r) \in U_r$, we have
\begin{equation}\label{trans}
\begin{split}
&e(\lambda_1) \sum_{k \ge -1} (s_1)_k (\alpha_2-\alpha_1)^{k+1}
L_{r-1}(\mu_2,\lambda_3,\ldots,\lambda_r; 
\alpha_2, \ldots, \alpha_r;
s_1+s_2+k+1,s_3,\ldots,s_r)\\
&= (1- e(\lambda_1)) L_r(\lambda_1,\ldots,\lambda_r; 
\alpha_1, \ldots, \alpha_r;
s_1,\ldots,s_r)\\
&+ \sum _{k\ge 0} (s_1)_k L_r(\lambda_1,\ldots,\lambda_r;
 \alpha_1, \ldots, \alpha_r;
s_1+k+1,s_2,\ldots,s_r),
\end{split}
\end{equation}
where $(s)_{-1}:=1$ and for $k \ge 0$,
$$(s)_k:=\frac{s\cdots(s+k)}{(k+1)!},$$
and the series on both sides of \eqref{trans}
converge normally on every compact subsets of $U_r$.
\end{thm}

If $\lambda_1=0$, we rewrite \eqref{trans} as,
\begin{equation}\label{trans2}
\begin{split}
&\sum_{k \ge -1} (s_1-1)_k (\alpha_2-\alpha_1)^{k+1}
L_{r-1}(\lambda_2,\lambda_3,\ldots,\lambda_r; \alpha_2, \ldots, \alpha_r;
s_1+s_2+k,s_3,\ldots,s_r)\\
&= \sum _{k\ge 0} (s_1-1)_k L_r(0,\lambda_2,\ldots,\lambda_r; 
\alpha_1, \ldots, \alpha_r;
s_1+k,s_2,\ldots,s_r).
\end{split}
\end{equation}

From now on, we will call the identities \eqref{trans}
and \eqref{trans2} as the generalised Ramanujan's identity
for the multiple Lerch zeta functions. 
In order to prove \thmref{gen-rama}, we introduce
another notation and prove some intermediate results.
For any $m \ge 0$, let
$$
U_r(m):=\{(s_1, \ldots, s_r) \in \C^r 
~|~ \Re(s_1+ \cdots + s_i) > i - m
~\text{ for all }~ 1\leq i \leq r \}.
$$
Note that $U_r = U_r(0)$. We first observe
that the series on the right hand side of
\eqref{lerch} is normally convergent on
compact subsets of $U_r$. For this
we need the following lemma from \cite{MSV}.

\begin{lem}\label{mzf}
For an integer $r \ge 1$, the family of functions
$$
\left ( \frac{1}{n_1^{s_1} \cdots n_r^{s_r}} \right )_{n_1>\cdots>n_r>0}
$$
converges normally on any compact subset of $U_r$.
\end{lem}

\begin{prop}\label{norm}
For an integer $r \ge 1$ and $\lambda_1, \ldots, \lambda_r,
\alpha_1, \ldots, \alpha_r  \in [0,1)$, the family of functions
$$
\left ( \frac{e(\lambda_1 n_1)\cdots e(\lambda_r n_r)}
{(n_1+\alpha_1)^{s_1} \cdots (n_r+\alpha_r)^{s_r}} \right )_{n_1>\cdots>n_r>0}
$$
converges normally on any compact subset of $U_r$.
\end{prop}

\begin{proof}
The proposition follows immediately from \lemref{mzf} as in $U_r$
$$
\left | \frac{e(\lambda_1 n_1)\cdots e(\lambda_r n_r)}
{(n_1+\alpha_1)^{s_1} \cdots (n_r+\alpha_r)^{s_r}} \right |
\le \left | \frac{1}{n_1^{s_1} \cdots n_r^{s_r}} \right |.
$$
\end{proof}

We further need the following propositions.

\begin{prop}\label{lerch-1}
Let $m \ge 0, r \ge 2$ be natural numbers and
$\lambda_1, \ldots, \lambda_r,
\alpha_1, \ldots, \alpha_r  \in [0,1)$. 
Then the family of functions
$$
\left ( (s_1)_k \frac{e(\lambda_1n_1) e(\lambda_2n_2)\cdots e(\lambda_r n_r)}
{(n_1+ \alpha_1)^{s_1+ k + 1} (n_2 + \alpha_2)^{s_2}
\cdots (n_r + \alpha_r)^{s_r}}\right )_{n_1> \cdots > n_r>0, \atop  k\ge m -1}
$$
is normally summable on compact subsets of $U_r(m)$. 
\end{prop}

\begin{proof}
Let $K$ be a compact subset of $U_r(m)$ and
$ S:=\sup_{(s_1,\ldots,s_r) \in K} |s_1|$. 
Since $r \ge 2$, one has $n_1 \ge 2$ and hence
for $k \ge m -1$ and $(s_1, \ldots, s_r) \in U_r(m)$, 
we have
$$
\left\|  (s_1)_k \frac{e(\lambda_1n_1) \cdots e(\lambda_r n_r)}
{(n_1+ \alpha_1)^{s_1+ k + 1}  (n_2 + \alpha_2)^{s_2}
\cdots (n_r + \alpha_r)^{s_r}} \right \|_K
\le \frac{(S)_k}{2^{k - m + 1}}
\left\| \frac{1}
{n_1^{s_1+m} n_2^{s_2} \cdots n_r^{s_r}} \right \|_K.
$$
Note that $(s_1, \ldots, s_r) \in U_r(m)$ if and only if
$(s_1+ m, s_2, \ldots, s_r) \in U_r$.
Now the proof of  \propref{lerch-1} follows 
from \lemref{mzf} and
the fact that the series
$$
\sum_{k \ge m -1}  \frac{(S)_k}{2^{k- m + 1}}  
$$ 
converges.
\end{proof}

\begin{prop}\label{lerch-2}
Let $ m \ge 0, r \ge 2$ be natural numbers and
$\lambda_1, \ldots, \lambda_r,
\alpha_1, \ldots, \alpha_r \in [0,1)$.
Then the family of functions
$$
\left((s_1)_k (\alpha_2 - \alpha_1)^{k+1}
\frac{e(\mu_2 n_2) e(\lambda_3 n_3) \cdots e(\lambda_r n_r)}
{(n_2+\alpha_2)^{s_1 + s_2 + k + 1} ~(n_3+\alpha_3)^{s_3}
\cdots (n_r + \alpha_r)^{s_r}}
\right)_{n_2 > \cdots > n_r >0, \atop k \ge m - 1}
$$
is normally summable on any compact subset of $U_r(m + 1)$
and hence on $U_r$. 
\end{prop}

\begin{proof}
As before, let $K$ be a compact subset of $U_r(m +1)$ and
$$ 
S:=\sup_{(s_1,\ldots,s_r) \in K} |s_1|.
$$ 
Then for $k \ge m-1, r \ge 2$ and $(s_1, \ldots, s_r) \in U_r(m)$, 
one has
\begin{align*}
&\left \| (s_1)_k 
\frac{(\alpha_2 - \alpha_1)^{k+1} e(\mu_2 n_2) 
e(\lambda_3 n_3) \cdots e(\lambda_r n_r)}
{(n_2+\alpha_2)^{s_1 + s_2 + k +1} (n_3+\alpha_3)^{s_3}
\cdots (n_r + \alpha_r)^{s_r}} \right \|\\
&\le  {(S)_k}{ (\alpha_2 - \alpha_1)^{k+1}}
\left\|  \frac{1}{n_2^{s_1 + s_2 + m } n_3^{s_3} \cdots n_r^{s_r}} \right \|.
\end{align*}
Note that
\begin{eqnarray*}
(s_1, \ldots,s_r) \in U_r(m + 1) 
&\implies&  
(s_1 + s_2, s_3, \ldots, s_r) \in U_{r-1}(m) \\ 
&\implies&
(s_1+ s_2 + m, s_3, \ldots, s_r) \in U_{r-1}. 
\end{eqnarray*} 
The proof now follows from \lemref{mzf} (for $(r-1)$)
and the fact that
$$
\sum_{k \ge m -1}  {(S)_k}{ (\alpha_2 - \alpha_1)^{k+1}} 
$$ 
converges, as $|\alpha_2 - \alpha_1| <1$.
\end{proof}

\begin{prop}\label{lerch-3}
Let $r \ge 2$ be an integer and
$\lambda_1, \ldots, \lambda_r, \alpha_1, \ldots, \alpha_r \in [0,1)$.
The family of functions
$$
\left(\frac{e(\lambda_1 n_1) \cdots e(\lambda_r n_r)}
{(n_1+\alpha_1-1)^{s_1}(n_2+\alpha_2)^{s_2} \cdots (n_r + \alpha_r)^{s_r}}
\right)_{n_1 > \cdots > n_r >0}
$$
is normally summable on any compact subset of $U_r$.
\end{prop}

\begin{proof}
Note that
$$
\left | \frac{e(\lambda_1 n_1) \cdots e(\lambda_r n_r)}
{(n_1+\alpha_1-1)^{s_1} (n_2+\alpha_2)^{s_2} \cdots (n_r + \alpha_r)^{s_r}} \right |
\le \left | \frac{1}{(n_1-1)^{s_1} n_2^{s_2} \cdots n_r^{s_r}} \right |.
$$
Also note that,
\begin{align*}
\left | \sum_{n_1 \ge n_2 +1} (n_1-1)^{-s_1} \right|
& \le n_2^{-\Re(s_1)} + \sum_{n_1 \ge n_2 +1} n_1^{-\Re(s_1)}\\
& \le n_2^{-\Re(s_1)} + \int_{n_2}^\infty x^{-\Re(s_1)}~dx\\
& = n_2^{-\Re(s_1)} + \frac{1}{\Re(s_1)-1} n_2^{1-\Re(s_1)}.
\end{align*}
The proof follows from \lemref{mzf}.
\end{proof}

\subsection{Proof of \thmref{gen-rama}}
We begin with the following identity which is valid for 
 any integer $n \ge 2$, any real number
$\alpha \ge 0$ and any complex number $s$:
\begin{equation}\label{trick}
(n+\alpha-1)^{-s}=\sum_{k \ge -1} (s)_k (n+\alpha)^{-s-k-1}.
\end{equation}
This identity is easily obtained by writing the left hand side as
$(n+\alpha)^{-s}(1-\frac{1}{n+\alpha})^{-s}$ and expanding
$(1-\frac{1}{n+\alpha})^{-s}$ as a Taylor series in $\frac{1}{n+\alpha}$.

In \eqref{trick} we replace $n,\alpha,s$ by $n_1,\alpha_1,s_1$
respectively and then multiply both sides by
$$
\frac{e(\lambda_1 n_1) \cdots e(\lambda_r n_r)}
{(n_2+\alpha_2)^{s_2} \cdots (n_r + \alpha_r)^{s_r}},
$$
and sum for $n_1>\cdots>n_r>0$.
Using \propref{lerch-3}, we get that
\begin{equation}\label{LHS-1}
\begin{split}
& \sum_{n_1 > \cdots > n_r >0} \frac{e(\lambda_1 n_1) \cdots e(\lambda_r n_r)}
{(n_1+\alpha_1-1)^{s_1} (n_2+\alpha_2)^{s_2} \cdots (n_r + \alpha_r)^{s_r}}\\
& = e(\lambda_1) \sum_{n_1 > \cdots > n_r >0}
\frac{e(\lambda_1 n_1) \cdots e(\lambda_r n_r)}
{(n_1+\alpha_1)^{s_1} (n_2+\alpha_2)^{s_2} \cdots (n_r + \alpha_r)^{s_r}}\\
& + e(\lambda_1) \sum_{n_2 > \cdots > n_r >0}
\frac{e(\mu_2 n_2) e(\lambda_3 n_3) \cdots e(\lambda_r n_r)}
{(n_2+\alpha_1)^{s_1} (n_2+\alpha_2)^{s_2} (n_3+\alpha_3)^{s_3}
\cdots (n_r + \alpha_r)^{s_r}}.
\end{split}
\end{equation}

Now,
$$
(n_2+\alpha_1)^{-s_1}
= \sum_{k \ge -1} (s_1)_k
(\alpha_2-\alpha_1)^{k+1} (n_2+\alpha_2)^{-s_1-k-1}.
$$
Hence using \propref{lerch-2} (for $m=0$), we obtain that
\begin{equation}\label{LHS-2}
\begin{split}
& \sum_{n_1 > \cdots > n_r >0}
\frac{e(\lambda_1 n_1) \cdots e(\lambda_r n_r)}
{(n_1+\alpha_1-1)^{s_1} (n_2+\alpha_2)^{s_2}
\cdots (n_r + \alpha_r)^{s_r}}\\
& = e(\lambda_1) L_r(\lambda_1,\ldots,\lambda_r; 
\alpha_1, \ldots, \alpha_r; s_1,\ldots,s_r)\\
& + e(\lambda_1) \sum_{k \ge -1} (s_1)_k (\alpha_2-\alpha_1)^{k+1}
L_{r-1}(\mu_2,\lambda_3,\ldots,\lambda_r; 
\alpha_2, \ldots, \alpha_r;
s_1+s_2+k+1,s_3,\ldots,s_r).
\end{split}
\end{equation}
On the other hand, using \eqref{trick} and
\propref{lerch-1} (for $m=0$), we get that
\begin{equation}\label{RHS}
\begin{split}
& \sum_{n_1 > \cdots > n_r >0}
\frac{e(\lambda_1 n_1) \cdots e(\lambda_r n_r)}
{(n_1+\alpha_1-1)^{s_1} (n_2+\alpha_2)^{s_2}
\cdots (n_r + \alpha_r)^{s_r}}\\
& = \sum _{k\ge -1} (s_1)_k L_r(\lambda_1,\ldots,\lambda_r;
\alpha_1, \ldots, \alpha_r; s_1+k+1,s_2,\ldots,s_r)
\end{split}
\end{equation}
Now equating the right hand sides of \eqref{LHS-2}
and \eqref{RHS} we deduce \eqref{trans}. This together with
\propref{lerch-1}, \propref{lerch-2} completes the proof.
\qed

\section{Proofs of \thmref{easy} and \thmref{multiple-Lerch}}

In this section, we use the generalised Ramanujan's
identities \eqref{trans} and \eqref{trans2}
to prove \thmref{easy} and \thmref{multiple-Lerch}.
We will prove these theorems by induction on depth $r$.
We assume that the multiple Lerch zeta function of depth
$(r-1)$ has already been extended to $\C^r$
and then by induction on $m \ge 1$ we extend the
multiple Lerch zeta function of depth $r$ to each of $U_r(m)$.
Since, $(U_r(m))_{m \ge 1}$ is an open covering
of $\C^r$ we will get our desired result.

\subsection{Proof of \thmref{easy}}
When $r=1$, then \thmref{easy} is true by \rmkref{easy-kn}.
Now let $r \ge 2$ and $\mu_i \not\in \Z$ for $1 \le i \le r$.
For any $m \ge 1$, we rewrite \eqref{trans} as
\begin{equation*}
\begin{split}
&e(\lambda_1) \sum_{k \ge m-2} (s_1)_k (\alpha_2-\alpha_1)^{k+1}
L_{r-1}(\mu_2,\lambda_3,\ldots,\lambda_r; 
\alpha_2, \ldots, \alpha_r;s_1+s_2+k+1,s_3,\ldots,s_r)\\
&+e(\lambda_1) \sum_{-1 \le k \le m-3} (s_1)_k (\alpha_2-\alpha_1)^{k+1}
L_{r-1}(\mu_2,\lambda_3,\ldots,\lambda_r; 
\alpha_2, \ldots, \alpha_r;s_1+s_2+k+1,s_3,\ldots,s_r)\\
&= (1- e(\lambda_1)) L_r(\lambda_1,\ldots,\lambda_r; 
\alpha_1, \ldots, \alpha_r;s_1,\ldots,s_r)\\
&+ \sum _{k\ge m-1} (s_1)_k L_r(\lambda_1,\ldots,\lambda_r;
 \alpha_1, \ldots, \alpha_r;s_1+k+1,s_2,\ldots,s_r)\\
&+ \sum _{0 \le k \le m-2} (s_1)_k L_r(\lambda_1,\ldots,\lambda_r;
 \alpha_1, \ldots, \alpha_r;s_1+k+1,s_2,\ldots,s_r).
\end{split}
\end{equation*}
Now by virtue of \propref{lerch-1}, \propref{lerch-2}
and the induction hypothesis for multiple Lerch zeta functions
of depth $(r-1)$, we see that all the $k$-sums in \eqref{trans}
are analytic in $U_r(1)$. Therefore \eqref{trans}
defines an analytic continuation of
$$
L_r(\lambda_1,\ldots,\lambda_r; 
\alpha_1, \ldots, \alpha_r;s_1,\ldots,s_r)
$$
as $e(\lambda_1)\neq 1$. Now suppose that
we have an analytic continuation of
$$
L_r(\lambda_1,\ldots,\lambda_r; 
\alpha_1, \ldots, \alpha_r;s_1,\ldots,s_r)
$$
to $U_r(m-1)$ which satisfies \eqref{trans} in $U_r(m-1)$.
Thus we get that the sum
$$
\sum _{0 \le k \le m-2} (s_1)_k L_r(\lambda_1,\ldots,\lambda_r;
\alpha_1, \ldots, \alpha_r;s_1+k+1,s_2,\ldots,s_r)
$$
is analytic in $U_r(m)$. Again we appeal to
\propref{lerch-1}, \propref{lerch-2}
and the induction hypothesis for multiple Lerch
zeta functions of depth $(r-1)$ to deduce that
all the $k$-sums in \eqref{trans} are analytic in $U_r(m)$.
Hence we obtain an analytic continuation of
$$
L_r(\lambda_1,\ldots,\lambda_r; 
\alpha_1, \ldots, \alpha_r;s_1,\ldots,s_r)
$$
to $U_r(m)$. Since, $(U_r(m))_{m \ge 1}$ is an open covering
of $\C^r$, this completes the proof. \qed

\subsection{Proof of \thmref{multiple-Lerch}.}

When $r=1$, then \thmref{multiple-Lerch} 
follows from \rmkref{easy-kn} if $\lambda_1 \not\in \Z$
and from \rmkref{multiple-Lerch-kn} if $\lambda_1 \in \Z$.
Now suppose $r \ge 2$ and \thmref{multiple-Lerch}
is true for multiple Lerch zeta function of depth $(r-1)$.

\subsection{Case 1 : $i_1 =1$}
In this case we have $\lambda_1 =0$ and hence
use \eqref{trans2}. We recall,
\begin{equation}\tag{\ref{trans2}}
\begin{split}
&\sum_{k \ge -1} (s_1-1)_k (\alpha_2-\alpha_1)^{k+1}
L_{r-1}(\lambda_2,\lambda_3,\ldots,\lambda_r; \alpha_2, \ldots, \alpha_r;
s_1+s_2+k,s_3,\ldots,s_r)\\
&= \sum _{k\ge 0} (s_1-1)_k L_r(0,\lambda_2,\ldots,\lambda_r;
\alpha_1, \ldots, \alpha_r; s_1+k,s_2,\ldots,s_r).
\end{split}
\end{equation}
To prove this case, we establish the meromorphic continuation of
$$
(s_1-1) L_r(0,\lambda_2,\ldots,\lambda_r;
\alpha_1, \ldots, \alpha_r; s_1,\ldots,s_r)
$$
to $\C^r$ using \eqref{trans2} and determine all its
possible singularities.

For any $m \ge 1$, we know by \propref{lerch-1}
and \propref{lerch-2} that the family of functions
$$
\left( (s_1-1)_k \frac{e(\lambda_2 n_2)\cdots e(\lambda_r n_r)}
{(n_1+\alpha_1)^{s_1+k}(n_2+\alpha_2)^{s_2} \cdots (n_r + \alpha_r)^{s_r}}
\right)_{n_1 > \cdots > n_r >0, \atop k \ge m} \\
$$
and
$$
\left( (s_1-1)_k (\alpha_2-\alpha_1)^{k+1}
\frac{e(\lambda_2 n_2)e(\lambda_3n_3)\cdots e(\lambda_r n_r)}
{(n_2+\alpha_2)^{s_1+s_2+k} \ (n_3+\alpha_3)^{s_3} \cdots
(n_r + \alpha_r)^{s_r}} \right)_{n_2 > \cdots > n_r >0, 
\atop k \ge m - 1}
$$
are normally summable on every compact subset of $U_r(m)$. 

Now for any $m \ge 1$, we rewrite \eqref{trans} as
\begin{equation*}
\begin{split}
&\sum_{k \ge m-1} (s_1-1)_k (\alpha_2-\alpha_1)^{k+1}
L_{r-1}(\lambda_2,\lambda_3,\ldots,\lambda_r; \alpha_2, \ldots, \alpha_r;
s_1+s_2+k,s_3,\ldots,s_r)\\
&+\sum_{-1 \le k \le m-2} (s_1-1)_k (\alpha_2-\alpha_1)^{k+1}
L_{r-1}(\lambda_2,\lambda_3,\ldots,\lambda_r; \alpha_2, \ldots, \alpha_r;
s_1+s_2+k,s_3,\ldots,s_r)\\
&= (s_1-1) L_r(0,\lambda_2,\ldots,\lambda_r;
\alpha_1, \ldots, \alpha_r; s_1,\ldots,s_r)\\
&+ \sum _{k\ge m} (s_1-1)_k L_r(0,\lambda_2,\ldots,\lambda_r;
\alpha_1, \ldots, \alpha_r; s_1+k,s_2,\ldots,s_r)\\
&+ \sum _{1\le k\le m-1} (s_1-1)_k L_r(0,\lambda_2,\ldots,\lambda_r;
\alpha_1, \ldots, \alpha_r; s_1+k,s_2,\ldots,s_r).
\end{split}
\end{equation*}
Using the above observation, we obtain that
both the infinite $k$-sums in the above equation
are analytic in $U_r(m)$. From the induction hypothesis
we deduce that the sum
$$
\sum_{-1 \le k \le m-2} (s_1-1)_k (\alpha_2-\alpha_1)^{k+1}
L_{r-1}(\lambda_2,\lambda_3,\ldots,\lambda_r; \alpha_2, \ldots, \alpha_r;
s_1+s_2+k,s_3,\ldots,s_r)
$$
has a meromorphic continuation to $\C^r$.
Now if we have that the function
$$
(s_1-1) L_r(0,\lambda_2,\ldots,\lambda_r;
\alpha_1, \ldots, \alpha_r; s_1,\ldots,s_r)
$$
has a meromorphic continuation to $U_r(m-1)$
for each $m \ge 1$, then we can deduce that the sum
$$
\sum _{1\le k\le m-1} (s_1-1)_k L_r(0,\lambda_2,\ldots,\lambda_r;
\alpha_1, \ldots, \alpha_r; s_1+k,s_2,\ldots,s_r)
$$
has a meromorphic continuation to $U_r(m)$ for each $m \ge 1$.
Therefore we obtain a meromorphic continuation of
$$
(s_1-1) L_r(0,\lambda_2,\ldots,\lambda_r;
\alpha_1, \ldots, \alpha_r; s_1,\ldots,s_r)
$$
to $U_r(m)$ by means of \eqref{trans2}.
Since, $(U_r(m))_{m \ge 1}$ is an open covering
of $\C^r$, we obtain a meromorphic continuation of
$$
(s_1-1) L_r(0,\lambda_2,\ldots,\lambda_r;
\alpha_1, \ldots, \alpha_r; s_1,\ldots,s_r)
$$
to $\C^r$. 

Now for the set of singularities, we see from \eqref{trans2}
that the singularities of
$$
(s_1-1) L_r(0,\lambda_2,\ldots,\lambda_r;
\alpha_1, \ldots, \alpha_r; s_1,\ldots,s_r)
$$
can only come from that of
$$
L_{r-1}(\lambda_2,\ldots,\lambda_r; \alpha_2, \ldots, \alpha_r;
s_1+s_2+k,s_3,\ldots,s_r)
$$
for all $k \ge -1$, and these singularities are known from the
induction hypothesis. Finally we deduce that
$$
(s_1-1) L_r(0,\lambda_2,\ldots,\lambda_r;
\alpha_1, \ldots, \alpha_r; s_1,\ldots,s_r)
$$
has only possible polar singularities along the hyperplanes
$$
H_{i_j, k} ~\text{ for  }~
2 \le j \le m \ \text{ with }~(i_j -k) \in \Z_{\le j}.
$$
This completes the proof of this case.

\subsection{Case 2 : $i_1 \neq 1$}
Since in this case the applicable generalised Ramanujan's identity
is \eqref{trans}, proof of this case follows exactly the line of
argument for the proof of \thmref{easy}. The only difference would be that
on each of $U_r(m)$ the depth $r$ multiple Lerch zeta function can only be
extended as a meromorphic function. This is because of 
the induction hypothesis which implies that the depth $(r-1)$
multiple Lerch zeta functions 
$$
L_{r-1}(\mu_2,\lambda_3,\ldots,\lambda_r; \alpha_2, \ldots, \alpha_r;
s_1+s_2+k+1,s_3,\ldots,s_r)
$$
for $k \ge -1$ can only be extended as meromorphic functions to $\C^r$.

Now for the set of singularities, we see from \eqref{trans}
that the singularities of
$$
L_r(\lambda_1,\ldots,\lambda_r;
\alpha_1, \ldots, \alpha_r; s_1,\ldots,s_r)
$$
can only come from that of
$$
L_{r-1}(\mu_2,\lambda_3,\ldots,\lambda_r; \alpha_2, \ldots, \alpha_r;
s_1+s_2+k+1,s_3,\ldots,s_r)
$$
for $k \ge -1$. These singularities are known from the
induction hypothesis and hence we deduce that
$$
L_r(\lambda_1,\ldots,\lambda_r;
\alpha_1, \ldots, \alpha_r; s_1,\ldots,s_r)
$$
has only possible polar singularities along the hyperplanes
$$
H_{i_j, k} ~\text{ for  }~
1 \le j \le m \ \text{ with }~(i_j -k) \in \Z_{\le j}.
$$
\qed

\section{Explicit computations of residues
for multiple Hurwitz zeta functions}

To get hold of the exact set of singularities we need
to calculate the residues of the multiple Lerch zeta
functions along its possible polar hyperplanes. For
a hyperplane $H_{i,k}$, by residue of
$$
L_r ( \lambda_1, \ldots, \lambda_r;
 \alpha_1, \ldots, \alpha_r ;  s_1, \ldots, s_r )
$$
along $H_{i,k}$ we mean the restriction to $H_{i,k}$
of the meromorphic function
$$
(s_1+\cdots+s_i-i+k) L_r ( \lambda_1, \ldots, \lambda_r;
 \alpha_1, \ldots, \alpha_r ;  s_1, \ldots, s_r ).
$$

It turns out that to study non-vanishing
of these residues one needs information about
zero sets of a family of polynomials with
two variables (see \rmkref{gen-Ber} below).
But for multiple Hurwitz zeta functions
we only have to deal with the family of
Bernoulli polynomials. As the zero set
of Bernoulli polynomials are well-studied
we just have enough information to
determine the exact set of singularities
of the multiple Hurwitz zeta functions.

In what follows, we obtain a computable expression
for residues of the multiple Hurwitz zeta functions.
Note that the applicable generalised Ramanujan's
identity in this case is \eqref{trans2}. Following
this process one can also obtain similar
expression for residues of the multiple Lerch zeta functions.
For brevity, we do not include this here.
We begin this section with some elementary
remarks about infinite triangular matrices.

Let $R$ be a commutative ring with unity. By ${\bf T}(R)$ we denote
the set of upper triangular matrices of type $\N \times \N$
with coefficients in $R$. Adding or multiplying such matrices involves
only finite sums, hence ${\bf T}(R)$ is a ring, and even an $R$-algebra.
The group of invertible elements of ${\bf T}(R)$ are the matrices
whose diagonal elements are invertible. Now let ${\bf P}$ be a matrix in
${\bf T}(R)$ with all diagonal elements equal to $0$, and
$f = \sum_{n \ge 0} a_n x^n \in R[[x]]$ be a formal power series,
then the series $\sum_{n \ge 0} a_n {\bf P}^n$ converges in ${\bf T}(R)$
and we denote its sum by $f({\bf P})$. For our purpose, we take
$R$ to be the field of rational fractions $\C(t)$ in one indeterminate
$t$ over $\C$.

Recall that from \thmref{gen-rama}, we get that
the multiple Hurwitz zeta function of
depth $r$ satisfy the following generalised
Ramanujan's identity:
\begin{equation}\label{tf-mhzf-2}
\begin{split}
& \sum_{k \ge -1} (s_1-1)_k \ (\alpha_2-\alpha_1)^{k+1} \
\z_{r-1}(s_1+s_2+k,s_3,\ldots,s_r;\alpha_2,\alpha_3,\ldots,\alpha_r)\\
&=\sum_{k\ge 0} (s_1-1)_k \
\z_r(s_1+k,s_2,\ldots,s_r;\alpha_1,\alpha_2,\ldots,\alpha_r),
\end{split}
\end{equation}
where both the above series of meromorphic functions
converge normally on all compact subsets of $\C^r$.
Formula \eqref{tf-mhzf-2}
together with the set of relations obtained by applying successively
the change of variable $s_1 \mapsto s_1+n$ 
for $n \ge 1$ to \eqref{tf-mhzf-2}, can be written as
\begin{equation}\label{mat-tf-mhzf}
\begin{split}
&{\bf A_2}(\alpha_2-\alpha_1;s_1-1)
{\bf V}_{r-1}( s_1+s_2-1,s_3,\ldots,s_r; \alpha_2, \ldots, \alpha_r)\\
&={\bf A_1}(s_1-1)
{\bf V}_r(s_1,\ldots,s_r; \alpha_1, \ldots, \alpha_r).
\end{split}
\end{equation}
Here for an indeterminate $t$, we have
\begin{equation}
\label{eqA1}
{\bf A_1}(t) := \left( \begin{array}{c c c c}
t & \frac{t(t+1)}{2!} & \frac{t(t+1)(t+2)}{3!} & \cdots\\
0 & t+1 & \frac{(t+1)(t+2)}{2!} & \cdots\\
0 & 0 & t+2 & \cdots\\
\vdots & \vdots & \vdots & \ddots
\end{array} \right),
\end{equation}
\begin{equation}\label{eqA2}
{\bf A_2}(\alpha_2-\alpha_1;t) := \left( \begin{array}{c c c c}
1 & t(\alpha_2-\alpha_1) &  \frac{t(t+1)}{2!}(\alpha_2-\alpha_1)^2 & \cdots \\
0 & 1 & (t+1)(\alpha_2-\alpha_1) & \cdots \\
0 & 0 & 1 &  \cdots \\
\vdots & \vdots & \vdots & \ddots
\end{array} \right)
\end{equation}
and
\begin{equation}\label{eqV-mhzf}
{\bf V}_r(s_1,\ldots,s_r; \alpha_1, \ldots, \alpha_r)
:= \left( \begin{array}{c}
\z_r(s_1,s_2,\ldots,s_r;\alpha_1, \ldots, \alpha_r)\\
\z_r(s_1+1,s_2,\ldots,s_r;\alpha_1, \ldots, \alpha_r)\\
\z_r(s_1+2,s_3\ldots,s_r;\alpha_1, \ldots, \alpha_r)\\
\vdots
\end{array} \right).
\end{equation}
Note that the matrix ${\bf A_1}(t)$ can be written as
$$
{\bf A_1}(t) = {\bf \Delta}(t) f({\bf M}(t+1)),
$$
where $f$ is the formal power series
$$
f(x):= \frac{e^x-1}{x} = \sum_{n \ge 0} \frac{x^n}{(n+1)!},
$$
and ${\bf \Delta}(t), {\bf M}(t) $ are as follows:
$$
{\bf \Delta}(t) := \left( \begin{array}{c c c c}
t & 0 & 0 & \cdots \\
0 & t+1 & 0 & \cdots \\
0 & 0 & t+2 & \cdots\\
\vdots & \vdots & \vdots & \ddots
\end{array} \right) \ \text{ and } \
{\bf M}(t) := \left( \begin{array}{c c c c}
0 & t & 0 & \cdots\\
0 & 0 & t+1 & \cdots\\
0 & 0 & 0 & \cdots\\
\vdots & \vdots & \vdots & \ddots
\end{array} \right).
$$
It is easy to see that ${\bf \Delta}(t), {\bf M}(t) $ satisfy the
following commuting relation:
\begin{equation}\label{delta-M}
{\bf \Delta}(t) {\bf M}(t+1) = {\bf M}(t){\bf \Delta}(t).
\end{equation}
Thus using \eqref{delta-M}, we have
$$
{\bf A_1}(t) = f({\bf M}(t)) {\bf \Delta}(t).
$$
Further, it is also possible to write that
$$
{\bf A_2}(\alpha_2-\alpha_1; t)=h({\bf M}(t)),
$$
where $h$ denotes the power series
$$
e^{(\alpha_2-\alpha_1)x} = \sum_{n \ge 0} (\alpha_2-\alpha_1)^n \frac{x^n}{n!}.
$$
Clearly the matrix ${\bf A_2}(\alpha_2-\alpha_1; t)$ is invertible 
and we see that
$$
{\bf A_2}(\alpha_2-\alpha_1;t)^{-1} {\bf A_1}(t)
= \frac{f}{h}({\bf M}(t)) \ {\bf \Delta}(t)
= {\bf \Delta}(t) \ \frac{f}{h}({\bf M}(t+1)).
$$
Hence the inverse of the matrix 
${\bf A_2}(\alpha_2-\alpha_1;t)^{-1} {\bf A_1}(t)$ is given by
$$
{\bf B}(\alpha_2-\alpha_1;t):= {\bf A_1}(t)^{-1} {\bf A_2}(\alpha_2-\alpha_1;t)
=\frac{h}{f}({\bf M}(t+1)) \ {\bf \Delta}(t)^{-1}
={\bf \Delta}(t)^{-1} \ \frac{h}{f}({\bf M}(t)),
$$
where $\frac{h}{f}$ is the exponential
generating series of the Bernoulli polynomials
evaluated at the point $(\alpha_2-\alpha_1)$, i.e.
$$
\frac{h}{f}(x)=\frac{x e^{(\alpha_2-\alpha_1)x}}{e^x-1}
= \sum_{n \ge 0} \frac{B_n(\alpha_2-\alpha_1)}{n!}x^n.
$$
More precisely, we have
\begin{equation} \label{eqB2}
{\bf B}(\alpha_2-\alpha_1;t)= \left( \begin{array}{c c c c c}
\frac{1}{t} & \frac{B_1(\alpha_2-\alpha_1)}{1!} &
\frac{(t+1)B_2(\alpha_2-\alpha_1)}{2!} &
\frac{(t+1)(t+2)B_3(\alpha_2-\alpha_1)}{3!} & \cdots \\
0 & \frac{1}{t+1} & \frac{B_1(\alpha_2-\alpha_1)}{1!} &
\frac{(t+2)B_2(\alpha_2-\alpha_1)}{2!} & \cdots \\
0 & 0 & \frac{1}{t+2} & \frac{B_1(\alpha_2-\alpha_1)}{1!} & \cdots \\
0 & 0 & 0 & \frac{1}{t+3} & \cdots\\
\vdots & \vdots & \vdots & \vdots & \ddots
\end{array} \right).
\end{equation}

However, we can not express the column vector
${\bf V}_r(s_1,\ldots,s_r; \alpha_1, \ldots, \alpha_r)$
as the product of the matrix ${\bf B}(\alpha_2-\alpha_1;s_1-1)$
and the column vector
${\bf V}_{r-1}( s_1+s_2-1,s_3,\ldots,s_r; \alpha_2, \ldots, \alpha_r)$.
This is because the infinite series involved in this product
are not convergent. To get around this difficulty
we perform a truncation process.

We first rewrite \eqref{mat-tf-mhzf} in the form
\begin{equation}\label{mat-tf-mhzf-2}
\begin{split}
&{\bf \Delta}(s_1-1)^{-1}
{\bf V}_{r-1}( s_1+s_2-1,s_3,\ldots,s_r; \alpha_2, \ldots, \alpha_r)\\
&=\frac{f}{h}({\bf M}(s_1))
{\bf V}_r(s_1,\ldots,s_r; \alpha_1, \ldots, \alpha_r).
\end{split}
\end{equation}
For notational convenience, let us denote $\frac{f}{h}({\bf M}(s_1))$ by
${\bf X}(s_1)$. 
We then choose an integer $q \ge 1$ and define
$$
I := \{ k  ~|~ 0 \le k \le q-1 \}
\phantom{m}\text{and} \phantom{m}
J := \{ k ~|~ k \ge q \}.
$$ 
Then we write our matrices as block matrices, for example
$$
{\bf X}(s_1) = \left( \begin{array}{c c}
{\bf X}^{II}(s_1) & {\bf X}^{IJ}(s_1)\\
{\bf 0}^{JI} & {\bf X}^{JJ}(s_1)
\end{array} \right).
$$
Hence from \eqref{mat-tf-mhzf-2} we get that
\begin{equation}
\label{dvxvxv}
\begin{split}
& {\bf \Delta}^{II}(s_1-1)^{-1}
{\bf V}_{r-1}^I(s_1+s_2-1,s_3,\ldots,s_r; \alpha_2, \ldots, \alpha_r)\\
=&\ {\bf X}^{II}(s_1) {\bf V}_r^I(s_1,\ldots,s_r; \alpha_1, \ldots, \alpha_r)
+ {\bf X}^{IJ}(s_1) {\bf V}_r^J(s_1,\ldots,s_r; \alpha_1, \ldots, \alpha_r).
\end{split}
\end{equation}

Since ${\bf X}^{II}(s_1)$ is a finite invertible square matrix, we have
$$
{\bf X}^{II}(s_1)^{-1} {\bf \Delta}^{II}(s_1-1)^{-1}
= {\bf B}^{II}(\alpha_2-\alpha_1;s_1-1).
$$
Therefore we deduce from \eqref{dvxvxv} that
\begin{equation}\label{vbvy}
\begin{split}
&{\bf V}_r^I(s_1,\ldots,s_r; \alpha_1, \ldots, \alpha_r)\\
& = {\bf B}^{II}(\alpha_2-\alpha_1;s_1-1)
{\bf V}_{r-1}^I(s_1+s_2-1,s_3,\ldots,s_r; \alpha_2, \ldots, \alpha_r)\\
&+{\bf Y}^I(s_1,\ldots,s_r; \alpha_1, \ldots, \alpha_r) ,
\end{split}
\end{equation}
where
\begin{equation}
\label{yxxv}
{\bf Y}^I(s_1,\ldots,s_r; \alpha_1, \ldots, \alpha_r) =  - {\bf X}^{II}(s_1)^{-1}
{\bf X}^{IJ}(s_1) {\bf V}_r^J(s_1,\ldots,s_r; \alpha_1, \ldots, \alpha_r).
\end{equation}

All the series of meromorphic functions involved in the products of
matrices in formulas \eqref{vbvy} and \eqref{yxxv} converge normally
on all compact subsets of $\C^r$. Moreover, all entries of the
matrices on the right hand side of \eqref{yxxv} are holomorphic on
the open set $U_r(q)$, translate of $U_r$ by $(-q,0,\ldots,0)$. Therefore
the entries of ${\bf Y}^I(s_1,\ldots,s_r; \alpha_1, \ldots, \alpha_r)$ are also
holomorphic in $U_r(q)$.
Let us denote $\xi_q(s_1,\ldots,s_r; \alpha_1, \ldots, \alpha_r)$
to be the first entry of the column
vector ${\bf Y}^I(s_1,\ldots,s_r; \alpha_1, \ldots, \alpha_r)$. Then we 
get from \eqref{vbvy} that
\begin{equation}\label{explicit-mhzf}
\begin{split}
&\zeta_r(s_1,\ldots,s_r; \alpha_1, \ldots, \alpha_r)\\
&=\ \frac{1}{s_1-1} \zeta_{r-1}(s_1+s_2-1,s_3,\ldots,s_r; \alpha_2, \ldots, \alpha_r)\\
&+ \sum_{k=0}^{q-2} \frac{s_1\cdots (s_1+k-1)}{(k+1)!} 
B_{k+1}(\alpha_2-\alpha_1)
\ \zeta_{r-1}(s_1+s_2+k,s_3,\ldots,s_r; \alpha_2, \ldots, \alpha_r)\\
&+ \xi_q(s_1,\ldots,s_r; \alpha_1, \ldots, \alpha_r),
\end{split}
\end{equation}
and $\xi_q(s_1,\ldots,s_r; \alpha_1, \ldots, \alpha_r)$
is holomorphic in the open set $U_r(q)$. 
In the above formula, whenever empty products and empty sums appear,
they are assumed to be $1$ and $0$ respectively. Formula \eqref{explicit-mhzf}
can also be obtained by using the Euler-Maclaurin summation
formula which was done in~\cite{AI}.

\begin{rmk}\label{gen-Ber} \rm
Matrix formulation of the generalised Ramanujan's
identity \eqref{trans2} would be similar as above.
If one wants to write down a matrix formulation
for the identity
\eqref{trans} one encounters a family of polynomials
$P_n(a,c)$ which are defined by the generating series
$$
\frac{e^{ax}}{e^x-c}=\sum_{n \ge 0} P_n(a,c) \frac{x^n}{n!}
$$
with $c \neq 1$.
\end{rmk}

We now observe that the following theorem can be deduced
as an immediate consequence of \thmref{multiple-Lerch}.

\begin{thm}\label{poles-mhzf}
The multiple Hurwitz zeta function of depth $r$ 
can be meromorphically continued to $\C^r$
with possible simple poles along the hyperplanes
$H_{1,0}$ and $H_{i,k}$, where $2 \le i \le r$ and
$k \ge 0$. It has at most simple poles along
each of these hyperplanes.
\end{thm}

To check if each $H_{i,k}$ is indeed a polar hyperplane, we compute
the residue of the multiple Hurwitz zeta function of depth $r$ along this
hyperplane using \eqref{vbvy} and \eqref{explicit-mhzf}.
Recall that it is defined as the restriction of the meromorphic function
$(s_1+\cdots+s_i-i+k) ~\zeta_r(s_1,\ldots,s_r; \alpha_1, \ldots, \alpha_r)$
to $H_{i,k}$.

\begin{thm}\label{residues-mhzf}
The residue of the multiple Hurwitz zeta function
$\zeta_{r}(s_1,\ldots,s_r;\alpha_1, \ldots, \alpha_r)$
along the hyperplane $H_{1,0}$ is the restriction of
$\zeta_{r-1}(s_2,\ldots,s_r;\alpha_2, \ldots, \alpha_r)$ 
to $H_{1,0}$ and its residue along the hyperplane $H_{i,k}$,
where $2 \le i \le r$ and $k \ge 0$, is the restriction to $H_{i,k}$
of the product of $\zeta_{r-i}(s_{i+1},\ldots,s_r;\alpha_{i+1}, \ldots, \alpha_r)$
with the $(0,k)^{\text{th}}$
entry of the matrix 
$$
\prod\limits_{d=1}^{i-1}
{\bf B}(\alpha_{d+1}-\alpha_d;s_1+\cdots+s_d-d).
$$
\end{thm}

\begin{proof}
Let $q \ge 1$ be an integer.  As in the proof of Theorem~\ref{poles-mhzf},
we know from \eqref{explicit-mhzf} that
$$
\zeta_r(s_1,\ldots,s_r;\alpha_1, \ldots, \alpha_r) -
\frac{1}{s_1-1} \ \zeta_{r-1}(s_1+s_2-1,s_3,\ldots,s_r;\alpha_2, \ldots, \alpha_r)
$$
has no pole along $H_{1,0}$ inside the open set $U_r(q)$. These open sets
cover $\C^r$. Hence the residue of $\zeta_r(s_1,\ldots,s_r;\alpha_1, \ldots, \alpha_r)$
along $H_{1,0}$ is the restriction to $H_{1,0}$ of the meromorphic function
$\zeta_{r-1}(s_1+s_2-1,s_3,\ldots,s_r;\alpha_2, \ldots, \alpha_r)$ or equivalently of
$\zeta_{r-1}(s_2,\ldots,s_r;\alpha_2, \ldots, \alpha_r)$. This proves the first part of
\thmref{residues-mhzf}.

Now let $i, k$ be integers with $2 \le i \le r$ and
$0 \le k < q$. Also let $I$ and  $J$ be as in \S4.4.
Now if one iterates $(i-1)$ times the formula \eqref{vbvy}, one gets
\begin{equation*}
\begin{split}
{\bf V}_r^I(s_1,\ldots,s_r; \alpha_1, \ldots, \alpha_r)
= & \ \left(\prod_{d=1}^{i-1}
{\bf B}^{II}(\alpha_{d+1}-\alpha_d;s_1+\cdots+s_d-d)\right)\\
&\times {\bf V}_{r-i+1}^I (s_1+\cdots +s_i-i+1,s_{i+1},\ldots,s_r;
\alpha_i, \ldots, \alpha_r)\\
&+ {\bf Y}^{i,I} (s_1,\ldots,s_r;\alpha_1, \ldots, \alpha_r),
\end{split}
\end{equation*}
where ${\bf Y}^{i,I} (s_1,\ldots,s_r;\alpha_1, \ldots, \alpha_r)$ is a column
matrix whose entries are finite
sums of products of rational functions in $s_1,\ldots,s_{i-1}$ with meromorphic
functions which are holomorphic in $U_r(q)$. These entries therefore have no
pole along the hyperplane $H_{i,k}$ in $U_r(q)$.
The entries of 
$$
\prod_{d=1}^{i-1} {\bf B}^{II}(\alpha_{d+1}-\alpha_d;s_1+\cdots+s_d-d)
$$ 
are rational functions in $s_1,\ldots,s_{i-1}$ and hence
have no poles along $H_{i,k}$.
It now follows from the induction hypothesis that the only entry 
of ${\bf V}_{r-i+1}^I (s_1+\cdots +s_i-i+1,s_{i+1},\ldots,s_r; \alpha_i, \ldots, \alpha_r)$
that can possibly have a pole along $H_{i,k}$ in $U_r(q)$ is
the one of index $k$, which is
$$
\zeta_{r-i +1} (s_1 + \ldots + s_i - i + k +1, s_{i +1}, \ldots, s_r;
\alpha_i, \ldots, \alpha_r).
$$
Its residue is the restriction of
$\zeta_{r-i}(s_{i+1},\ldots,s_r; \alpha_{i+1}, \ldots, \alpha_r)$ to
$H_{i,k} \cap U_r(q)$, where $2 \le i \le r$ and $0~\le~k < q$.
Since the open sets $U_r(q)$ for $q > k$
cover $\C^r$, the residue of $\zeta_r(s_1,\ldots,s_r; \alpha_1, \ldots, \alpha_r)$ 
along $H_{i,k}$ is the restriction to $H_{i,k}$ of the product of the
$(0,k)^{\text{th}}$ entry of the matrix
$$
\prod_{d=1}^{i-1} {\bf B}(\alpha_{d+1}-\alpha_d;s_1+\cdots+s_d-d)
$$
with $\zeta_{r-i}(s_{i+1},\ldots,s_r; \alpha_{i+1}, \ldots, \alpha_r)$.
This proves the last part of \thmref{residues-mhzf}.
\end{proof}

\section{Proof of \thmref{multiple-Hurwitz}}

\subsection{Zeros of Bernoulli polynomials}
The information about the exact set of poles of 
multiple Hurwitz zeta functions in \thmref{multiple-Hurwitz} 
requires knowledge about the
zeros of the Bernoulli polynomials.
In this section, we discuss those properties of the zeros
of the Bernoulli polynomials
which are relevant to our study.

Recall that the Bernoulli polynomials $B_n(t)$ are defined by
$$
\sum_{n \ge 0} B_n(t) \frac{x^n}{n!} = \frac{x e^{tx}}{e^x -1}.
$$
We have the following theorem by Brillhart \cite{JB} and
Dilcher \cite{KD} about the zeros of Bernoulli polynomials.

\begin{thm}[Brillhart-Dilcher]\label{Brill-Dil}
Bernoulli polynomials do not have multiple roots.
\end{thm}

This theorem was first proved for the odd Bernoulli
polynomials by Brillhart
\cite{JB} and later extended for the even
Bernoulli polynomials by Dilcher \cite{KD}.
\thmref{Brill-Dil}
amounts to say that the Bernoulli polynomials $B_{n+1}(t)$ and
$B_n (t)$ are relatively prime as they satisfy the relation
$$
B_{n+1}'(t) = (n+1) B_n(t) \ \text{ for all } n \ge 1.
$$
where $B'_{n+1}(t)$ denotes the derivative of the polynomial $B_{n+1}(t)$.
With the theorem of Brillhart and Dilcher in place we can now
describe the exact set of singularities of the multiple zeta
functions. For that it is convenient to have some intermediate
lemmas in place.

\subsection{Some intermediate lemmas}

\begin{lem}\label{two-prod}
Let $x,y$ be two indeterminate and the matrix 
$\bf{B}$ be as in \eqref{eqB2}. 
Then all the entries in the first row of the matrix
$$
{\bf B}(\beta - \alpha; x) \ {\bf B}(\gamma - \beta; y),
$$
where $0 \le \alpha, \beta, \gamma< 1$, are non-zero rational
functions in $x,y$ with coefficients in $\R$.
\end{lem}

\begin{proof}
Since entries of these matrices are indexed by $\N \times \N$, the entries
of the first row are written as $(0,k)^{\text{th}}$ entry
for $k \ge 0$. Let us denote the $(0,k)^{\text{th}}$ entry
by $a_{0,k}$. Then we have the following formula:
$$
x (y+k) \ a_{0,k} = \sum_{i=0}^{k} (x)_{i-1} (y+i+1)_{k-i-1}
B_i(\beta-\alpha) \ B_{k-i}(\gamma - \beta)
$$
for all $k \ge 0$. As the Bernoulli polynomial $B_0(t)$ is equal to $1$,
we get $a_{0,0}= \frac{1}{xy}$ and hence non-zero.
For $k \ge 1$, we first note that the set of polynomials
$$
P:=\{(x)_{i-1} (y+i+1)_{k-i-1} : 0 \le i \le k \}
$$
is linearly independent over $\R$.
 
Now suppose that $B_1(\beta-\alpha) \neq 0$. We know by
\thmref{Brill-Dil} that at least one of $B_k(\gamma - \beta)$
and $B_{k-1}(\gamma - \beta)$ is non-zero. It now follows 
from the linear independence of the set of polynomials in $P$
that $a_{0,k} \neq 0$.

Next suppose that $B_1(\beta-\alpha) = 0$, i.e. $\beta-\alpha = 1/2$.
Then $\gamma-\beta \neq 1/2$ as $0 \le \alpha, \gamma< 1$.
Hence $B_1(\gamma-\beta) \neq 0$. Again by
\thmref{Brill-Dil}, we know that at least one of $B_k(\beta-\alpha)$
and  $B_{k-1}(\beta-\alpha)$ is non-zero. Now by
linear independence of the set of polynomials in $P$, we get
$a_{0,k} \neq 0$. This completes the proof of \lemref{two-prod}.
\end{proof}

\begin{lem}\label{any-prod}
Let $n \ge 0$ be an integer and $x,x_1, \ldots, x_n$ be $(n+1)$ 
indeterminate. Let ${\bf D}$ be an infinite square matrix whose
entries are indexed by $\N \times \N$ and is in the ring
$\R(x_1, \ldots, x_n)$. Further, suppose that all the entries in the
first row of ${\bf D}$ are non-zero. Then for any $\alpha, \beta \in \R$,
all the entries in the first row of the matrix ${\bf D}{\bf B}(\beta - \alpha; x)$
are non-zero, where the matrix $\bf{B}$ be as in \eqref{eqB2}.
\end{lem}

\begin{proof}
We first note that each column of
${\bf B}(\beta - \alpha; x)$ has at least one non-zero entry and
the non-zero entries of each of these columns are linearly independent over $\R$
as rational functions in $x$ with coefficients in $\R$.
Since all the entries in the first row of ${\bf D}$ are non-zero,
the proof is complete by the above observation.
\end{proof}

We are now ready to prove \thmref{multiple-Hurwitz}.

\subsection{Proof of \thmref{multiple-Hurwitz}}
When $1 \le i \le r$ and $k \ge 0$, the restriction of
$$
\zeta_{r-i}(s_{i+1},\ldots,s_r,\alpha_{i+1}, \ldots, \alpha_r)
$$
to $H_{i,k}$ is a non-zero meromorphic function.
Hence in order to prove \thmref{multiple-Hurwitz},
we need to show that when $2 \le i \le r$ and $k \ge 0$, 
the $(0,k)^{\text{th}}$ entry of the matrix
$$
\prod\limits_{d=1}^{i-1}
{\bf B}(\alpha_{d+1}-\alpha_d;s_1+\cdots+s_d-d)
$$
is identically zero if and only if $i=2, \ k \in J$.
By changing co-ordinates, the above statement is 
equivalent to say that when
$t_1,\ldots,t_{i-1}$ are indeterminate, the $(0,k)^{\text{th}}$ entry
of the matrix 
$$
\prod\limits_{d=1}^{i-1}{\bf B}(\alpha_{d+1}-\alpha_d;t_d)
$$
is non-zero in $\R(t_1,\ldots,t_{i-1})$ except
when $i=2$ and $k \in J$.

For $i=2$, our matrix is ${\bf B}(\alpha_2-\alpha_1;t_1)$ and
hence our assertion follows immediately. Now assume that $i \ge 3$.
By \lemref{two-prod}, we know that  all the entries in the first row
of the matrix 
$$
{\bf B}(\alpha_2-\alpha_1;t_1) {\bf B}(\alpha_3-\alpha_2;t_2)
$$
is non-zero in $\R(t_1,t_2)$. Hence the theorem follows
from \lemref{two-prod} if $i=3$ and from repeated
application of \lemref{any-prod} if $i > 3$. \qed

\subsection{A particular case}

\thmref{multiple-Hurwitz} shows that precise knowledge 
about zeros of Bernoulli polynomials determines
the exact set of singularities of the multiple Hurwitz zeta
functions. Now we have precise knowledge about the
rational zeros of the Bernoulli polynomials due to Inkeri \cite{KI}.

\begin{thm}[Inkeri]\label{Inkeri}
The rational zeros of a Bernoulli polynomial $B_n (t)$
can only be $0, 1/2$ and $1$. This happens only when
$n$ is odd and precisely in the following cases:
\begin{enumerate}
\item $B_n(0)=B_n(1)=0$ for all odd $n \ge 3$,
\item $B_n(1/2)=0$ for all odd $n \ge 1$.
\end{enumerate}
\end{thm}

Using \thmref{Inkeri}, we deduce the following corollary
of \thmref{multiple-Hurwitz}. A particular case of this
corollary,  namely when
$\alpha_i \in \Q$ for all $1 \le i \le r$, was proved in \cite{AI}. 

\begin{cor}\label{special-mhzf}
If $\alpha_2-\alpha_1=0$, then the exact set of singularities of the
multiple Hurwitz zeta function
$\zeta_{r}(s_1,\ldots,s_r;\alpha_1, \ldots, \alpha_r)$ is given by the
hyperplanes
$$
H_{1,0}, H_{2,1}, H_{2,2k} \ \text{ and }\
H_{i,k}  \ \text{ for all } \ k \ge 0\ \text{ and } \ 3 \le i \le r.
$$
If $\alpha_2-\alpha_1=1/2$, then the exact set of singularities of the
multiple Hurwitz zeta function
$\zeta_{r}(s_1,\ldots,s_r;\alpha_1, \ldots, \alpha_r)$ is given by the
hyperplanes
$$
H_{1,0}, H_{2,2k} \ \text{ and }\
H_{i,k}  \ \text{ for all } \ k \ge 0\ \text{ and } \ 3 \le i \le r.
$$
If $\alpha_2-\alpha_1$ is a rational number $\neq 0, 1/2$,
then the exact set of singularities of the
multiple Hurwitz zeta function
$\zeta_{r}(s_1,\ldots,s_r;\alpha_1, \ldots, \alpha_r)$ is given by the
hyperplanes
$$
H_{1,0} \ \text{ and }\
H_{i,k}  \ \text{ for all } \ k \ge 0\ \text{ and } \ 2 \le i \le r.
$$
\end{cor}

\bigskip

\noindent
{\bf Acknowledgements.}
Both the authors would like to thank the Institute of Mathematical Science
(IMSc) where this work was done. Part of this work was supported by 
Number Theory plan project of DAE and a SERB grant.  
First author would also like to thank ICTP 
where part of the final draft was written.
The second author would like to thank the organisers
K. Soundararajan and J-M Deshouillers of a conference
in number theory held at IMSc during December 14th-18th, 2015
where he was allowed to present this work. We thank 
the referee for bringing a number of references to our notice.

\end{document}